\documentclass{llncs}
\usepackage{amssymb,amsmath}
\usepackage{multicol}
\usepackage{graphicx}
\usepackage{epsfig}
\usepackage{amsfonts}
\usepackage{fullpage}
\newtheorem{assumption}{Assumption}

\newcommand{\ignore}[1]{}

\newtheorem{alphatheorem}{Theorem}

\title{A Coloring Algorithm for Triangle-Free Graphs}
\author{Mohammad Shoaib Jamall
\protect\footnotemark[1]\protect\footnotetext[1]{Research supported in part by a Reuben H. Fleet Foundation Fellowship, 
and an ARCS Foundation Fellowship.}}
\institute{Department of Mathematics, UC San Diego \\
\email{mjamall@math.ucsd.edu}}

\begin{document}

\maketitle

\begin{abstract}
We give a randomized algorithm that
properly colors the vertices of a triangle-free graph $G$ on $n$ 
vertices using $O(\Delta(G)/\log \Delta(G))$ colors,
where $\Delta(G)$ is the maximum degree of $G$.
The algorithm takes $O(n \Delta^2(G) \log \Delta(G))$ time and succeeds with high probability, 
provided  $\Delta(G)$ is greater than $log^{1+\epsilon} n$ for a positive constant $\epsilon$. 
The number of colors is best possible up to a constant factor for triangle-free graphs.
As a result this gives an algorithmic proof for a sharp upper bound of the chromatic number of a triangle-free graph,
the existence of which was previously established by Kim and Johansson respectively.
%The existence of such a coloring was proved by J. H. Kim and A. Johansson.
%Our research is a natural continuation of work done by J. H. Kim and A. Johansson to show that such colorings exist.
%Our method is stronger than earlier algorithms which 
%achieved the same on regular graphs with girth greater than 4.
%The existence of such a coloring was proved using the semi-random method.
%Our algorithm and analysis combine ideas from the proof with a generalization of 
%Markov's inequality;
%we hope this technique will be a guide around some of the obstacles to obtaining
%efficient algorithms from such proofs.
%\marginpar{ this is asymptotically best possible?}
\end{abstract}

\section{Introduction}
% speak about experiments already done
% give intuition
% mention: that the algorithm is simple, but analysis is long perhaps 
% because of lack of proper established notations for the objects considered
% mention in intro: sparse coloring of K4 free graph is a big open problem
% mention desire to obtain simpler proof
%References: PG & CCJ
%acknowledge ARCS support
%References: PG & CCJ
%\section{A Summary of our Results}
%To Be Completed...
A \emph{proper vertex coloring} of a graph is an assignment of colors to all vertices
such that adjacent vertices have distinct colors.
%In this paper, 
%we give a randomized algorithm for properly coloring a triangle-free graph $G$ on $n$
%vertices with $\Delta(G) \geq \log^{1+\epsilon} n$ for a positive constant $\epsilon$,
%using $O(\Delta(G) / \log \Delta(G))$ colors in $O(n \Delta^2(G) \log \Delta(G))$ time.
%The probability of failure is $O(1/n)$.
The \emph{chromatic number} $\chi(G)$ of a graph $G$ is the minimum number of colors
required for a proper vertex coloring. 
Finding the chromatic number of a graph is NP-Hard \cite{GJ79}. 
Approximating it to within a polynomial ratio is also hard \cite{K01}.
For general graphs, $\Delta(G) + 1$ is a trivial upper bound.
Brooks' Theorem \cite{B41} shows that $\chi(G)$ can be $\Delta(G) + 1$ only if 
$G$ has a component which is either a complete subgraph or an odd cycle.

A natural question is:
can this bound be improved for graphs without large complete subgraphs?
In $1968$, 
Vizing \cite{V68} had asked what the best possible upper bound for the 
chromatic number of a triangle-free graph was.
Borodin and Kostochka \cite{BK77}, Catalin \cite{C78}, and Lawrence \cite{L78}
independently made progress in this direction;
they showed that for a $K_4$-free graph,
$\chi(G) \leq 3(\Delta(G) + 2)/4$.
On the other hand, Kostochka and Masurova \cite{KM77}, 
and Bollob{\'a}s \cite{B78} separately showed that there are graphs of arbitrarily large 
\emph{girth}(length of a shortest cycle) with
$\chi(G)$ of order $\Delta(G)/\log \Delta(G)$.
%Could we reduce this gap between $\Omega(\Delta(G) / \log \Delta(G))$ and 
%$O(\Delta(G))$, and improve the upper bound on the chromatic number of triangle-free graphs?
%
%The \emph{independence number} $\alpha(G)$ of a graph is the size of the maximum set 
%of vertices no two of which are adjacent.
%Ajtai, Koml{\'o}s and Szemer{\'e}di \cite{AKS81} proved that for a triangle-free graph $G$,
%\begin{equation*}
%\alpha(G) = \Omega(n \frac{\log t(G)}{t(G)})
%\end{equation*}
%where $t(G)$ is the average degree of the graph.
%This bound and the close relationship between $\chi(G)$ and $\alpha(G)$---for example, 
%it is well known that 
%$\chi(G) \geq n/\alpha(G)$---suggests that the correct upper bound may be 
%$O(\Delta(G) /  \log \Delta(G))$.

Further progress was made using the \emph{semi-random method}
to show that the chromatic number of graphs with large girth is $O(\Delta(G) / \log \Delta(G))$.
%Kim for graphs with girth greater than $4$, 
%and Johansson for triangle-free graphs(as mentioned earlier).
This technique,
also known as the \emph{pseudo-random method}, 
or the \emph{R{\"o}dl nibble},
appeared first in Ajtai, Koml{\'o}s and Szemer{\'e}di \cite{AKS81}
and was applied to problems in hypergraph packings,
Ramsey theory, colorings, and list colorings \cite{FR85,K92,K96,K2_95,PS89}.
In general, given a set $S_1$, the goal is to show that there is an object in $S_1$ with a desired property $\mathcal{P}$.
This is done by locating a sequence of non-empty subsets 
$S_1 \supseteq \dots \supseteq S_{\tau}$ with $S_{\tau}$ having property $\mathcal{P}$.
A randomized algorithm is applied to $S_t$, which guarantees that
$S_{t+1}$ will be obtained with some non-zero(often small) probability.
For upper bounds on chromatic number, 
the semi-random method is used to prove the existence of a proper coloring with a limited number of colors.
It does not give an efficient probabilistic algorithm.

In $1995$, Kim \cite{K95} proved that
\begin{equation*}
\chi(G) \leq (1+o(1))\frac{\Delta(G)}{\log \Delta(G)}
\end{equation*}
when $G$ has girth greater than $4$.
Later on, Johansson \cite{J96} showed that 
\begin{equation*}
\chi(G) \leq O(\frac{\Delta(G)}{\log \Delta(G)})
\end{equation*}
when $G$ is a triangle-free graph(girth greater than $3$).
Both Kim and Johansson used the semi-random method.
Alon, Krivelevich and Sudakov \cite{AKS99}, and 
Vu \cite{V02} extended the method of Johansson to prove bounds on the chromatic number 
for graphs in which no subgraph on the set of all neighbors of a vertex has too many edges.

Grable and Panconesi \cite{GP00} 
gave an algorithm for properly coloring a $\Delta$-regular graph with girth greater than 
$4$ and $\Delta \geq \log^{1+\epsilon} n$
for a positive constant $\epsilon$.
The procedure uses $O(\Delta / \log \Delta)$ colors and has polynomial running time.
This is a constructive version of the existential proof of Kim 
with extra conditions on the degree of the input graph.
Section $5$ of \cite{GP00} asserts that the algorithm can be extended to triangle-free graphs.
%their sequence of graph colorings $G=G_1,\dots,G_{\tau}$, 
%where graph $G$ is the uncolored input and graph $G_{\tau}$ is the final colored output,
%does not have the claimed properties,
%and their extension to triangle-free graphs does not work.
In Section \ref{sec:sketch} we will see a counter-example that shows that their analysis  does not work on all triangle-free graphs.
In particular, as will be seen, the analysis fails on a complete bipartite graph.

In this paper we give a method 
%with a simple but novel generalization of Markov's inequality.
for coloring the larger class of triangle-free graphs.
Our randomized algorithm properly colors a triangle-free graph $G$ on $n$ vertices with 
$\Delta(G) \geq \log^{1+\epsilon} n$ for a positive constant $\epsilon$,
using $O(\Delta(G) / \log \Delta(G))$ colors in $O(n \Delta^2(G) \log \Delta(G))$ time.
The probability of failure goes to $0$ as $n$ becomes large.

To analyze our iterative algorithm we identify a collection of random variables.
The expected changes to these random variables after a round of the algorithm are written in terms of the values of the random variables before the round.
We thus obtain a set of recurrence relations and prove that our random variables are concentrated around the solutions to the recurrence relations.

Our method of analysis resembles that of Kim, and Grable et al.
However our algorithm and collection of random variables are different.
Also related to our analysis technique is Achlioptas and Molloy's study of the performance of a coloring algorithm on random graphs \cite{AM97}.
They setup recurrence relations as we do,
and then use the so-called differential equation method for random graph processes \cite{W95}.

We describe our algorithm in Section \ref{sec:pseudo}.
Section \ref{sec:sketch} contains motivation, which is followed by a formal description of the algorithm in Section \ref{sec:algorithm_details}.
%Section \ref{sec:sketch} contains motivation,
%which can be skipped by the reader who desires to know the algorithm details directly.
%Section \ref{sec:algorithm_details} contains a detailed description of the algorithm.
We give an outline of the analysis in Section \ref{sec:outline}.
Section \ref{sec:prelim} contains some useful lemmas which we are used 
in Section \ref{sec:details} to give details of the analysis.
\section{An Iterative Algorithm for Coloring a Graph}
\label{sec:pseudo}
Our algorithm takes as input a triangle-free graph $G$ on $n$ vertices,
its maximum degree $\Delta$,
and the number of colors to use $\Delta / k$ where $k$ is a positive number.
It goes through rounds and assigns colors to more vertices each round.
Initially all vertices are \emph{uncolored}(no color assigned),
at the end we have a proper vertex coloring of $G$ with some probability.

\begin{definition} Let $t$ be a natural number. 
We define the following: \\%{\ } \\ \\
\begin{tabular}{ p{1.3cm} p{14cm} } 
$G_{t}$ & 
The graph induced on $G$ by the vertices that are uncolored at the beginning of round $t$. \\
$N_{t}(u)$ & 
The set of vertices adjacent to vertex $u$ in $G_t$. 
That is, the set of uncolored neighbors of $u$ at the beginning of round $t$. \\
$S_t(u)$ & The list of colors that may be assigned to vertex 
$u$ in round $t$, also called the \emph{palette} of $u$.
For all $u$ in $V(G)$,
\begin{equation*}
S_0(u) = \{1, \dots, \Delta/k\}.
\end{equation*} 
\\
$D_{t}(u,c)$ & The set of vertices adjacent to $u$ 
that may be assigned color $c$ in round $t$. That is,
\begin{equation*}
D_t(u,c) := \{v \in N_t(u) | c \in S_t(v)\}.
\end{equation*} 
\end{tabular}
\end{definition}
\begin{samepage}
It will be useful to define variables for the sizes of the sets $N_t(u)$, $S_t(u)$, and $D_t(u,c)$.
\begin{definition} {\ } 
\begin{align*}
\eta_t(u) &= |N_t(u)| \\
s_t(u) &= |S_t(u)| \\
d_t(u,c) &= |D_t(u,c)|
\end{align*}
\end{definition}
\end{samepage}
Observe that for every round $t$, vertex $u$ in $V(G)$, and color $c$ in $\{1, \dots, \Delta/k\}$,
\begin{equation*}
d_t(u,c) , \eta_t(u) \leq \Delta, \;\; s_0(u) = \Delta/k, \;\; s_t(u) \leq \Delta/k.
\end{equation*}

\subsection{A Sketch of the Algorithm and the Ideas behind its Analysis}
\label{sec:sketch}
We start by considering an algorithm that works for $\Delta$-regular graphs
with girth greater than $4$. The rounds of this algorithm will go through two stages.
The number of colors in the palette of each uncolored vertex will be greater than the
number of uncolored adjacent vertices at the end of the second stage with some probability.
Then the partial coloring of the graph can be completed easily to give a proper coloring.

Let $(\eta_t)$, $(d_t)$, and $(s_t)$ be sequences defined recursively as \\
\begin{center}
\begin{tabular}{l l p{1cm} l l}
$\eta_0$ 	& 	$:= \Delta$		&	&	$\eta_{t+1}$	& 	$:= \eta_t (1- c_1 \frac {s_t}{d_t})$ \\
$d_0$ 	& 	$:= \Delta$		&	&	$d_{t+1}$		& 	$:= d_t (1- c_1 \frac {s_t}{d_t}) c_2$ \\
$s_0$ 	& 	$:= \Delta/k$		&	&	$s_{t+1}$		& 	$:= s_t c_2$ \\
\end{tabular}
\end{center}
where $c_1$ and $c_2$ are constants between $0$ and $1$,
which are determined by the analysis of the algorithm.

\begin{description}
\item[First Stage]
{\ }
\begin{center}
\framebox{
\begin{minipage}{15cm}
{\bf Repeat at every round $t$, until $d_{t}/s_t < 1$}
\\
{\it Phase I - Coloring Attempt }
\\
For each vertex $u$ in $G_t$: \\
\hspace*{.5cm}
	{\it Awake} vertex $u$ with probabilitity $s_{t}/d_{t}$. \\
\hspace*{.5cm}
	If {\it awake}, assign to $u$ a color chosen from $S_t(u)$ uniformly at random. \\

{\it Phase II - Conflict Resolution}
\\
For each vertex $u$ in $G_t$: \\
\hspace*{.5cm}
	If {\it u is awake}, \\
\hspace*{1cm}
		uncolor $u$ if an adjacent vertex is assigned the same color in the coloring attempt phase. \\
\hspace*{.5cm}
	Remove from $S_t(u)$, all colors assigned to adjacent vertices. \\
\hspace*{.5cm}
	$S_{t+1}(u) = S_t(u).$ \\
{\bf end repeat}
\end{minipage}}
\end{center}
{\ } \\
Observe that $d_0 / s_0 = k$ and
\begin{equation*}
\frac{d_{t+1}}{s_{t+1}} = \frac {d_t}{s_t} - c_1.
\end{equation*}
In $O(k)$ rounds $d_t / s_t$ will be less that $1$, and this marks the end of the first stage.
\end{description}
\begin{description}
\item[Second Stage]
We change the recursive equations for $\eta_t$, $d_t$, and $s_t$. \\
\begin{center}
\begin{tabular}{l l}
$\eta_{t+1}$	& 	$:= \eta_t (1- c_1 e^{c_2 d_t / s_t})$ \\
$d_{t+1}$		& 	$:= d_t (1- c_1 e^{c_2 d_t / s_t}) e^{- c_2 d_t / s_t}$ \\
$s_{t+1}$		& 	$:= s_t e^{- c_2 d_t / s_t}$ \\
\end{tabular}
\end{center}
All uncolored vertices are woken up at every round.
In this stage, $\eta_t$ decreases much faster than $s_t$
and the repeat-until loop is repeated until 
$$\eta_t < s_t.$$
Otherwise the second stage is the same as the first.

\end{description}

\noindent The two stage algorithm is derived from Grable et al.
Their analysis tells us that if graph $G$ has girth greater than $4$ and 
$\Delta \geq \log^{1+\epsilon} n$ for some positive constant $\epsilon$,
then there are constants $c_1$ and $c_2$ less than $1$ such that at each round $t$,
$\forall u \in V(G_t), \forall c \in S_t(u)$
\begin{equation}
\label{eq:inv_sk}
\eta_t(u) = \eta_t(1+o(1)), \;\;\; s_t(u) = s_t(1+o(1)), \;\;\; d_t(u,c) = d_t(1+o(1))
\end{equation}
with probability $1-o(1)$.
The equations above imply that at the end of the second stage $s_t(u) > \eta_t(u)$
for all uncolored vertices $u$ with probability approaching $1$ as $n$ approaching $\infty$.

The change we have made to Grable et al. so far,
is that we remove all colors temporarily assigned to neighbors of a vertex from its palette,
instead of removing only those assigned permanently.
This simple but powerful idea, adapted from Kim \cite{K95},
will waste a few colors from the palettes, and 
simplify our algorithm and its analysis significantly.

\subsubsection{A counter-example to demonstrate that Grable et al. does not work on graphs with girth greater than $3$.}
The analysis for the above algorithm is probabilistic and proves the property in equation 
\eqref{eq:inv_sk} by induction,
showing concentration of the variables around their expectations.
It fails for graphs with $4$-cycles. An example illustrates why:
Consider a vertex $u$ whose $2$-neighborhood, 
the graph induced by vertices within distance $2$ of $u$,
is the complete bipartite graph 
$K_{\Delta,\Delta}$ with partitions $X$ and $Y$.
Suppose that $u$ and another vertex $v$ are in $X$.
If $v$ is colored with $c$ in round $0$ while $u$ remains uncolored,
then the set $D_{1}(u,c) = \emptyset$; this violates equation \eqref{eq:inv_sk} since $d_1 \geq 1$ 
if for example $k \geq 2$ and $\Delta \geq 2/c_2$.
So, when the graph has $4$-cycles, 
$d_{t+1}(u,c)$ is not necessarily concentrated around its expectation given 
the state of the algorithm at the beginning of round $t$.
Maintaining equation \eqref{eq:inv_sk} is crucial for the proof in Grable et al., 
and this violation is the error we mentioned in the introduction;
their algorithm and analysis do not work for triangle-free graphs in general!

We must modify the algorithm in two more ways.
\begin{description}
\item[First Modification: A technique for coloring graphs with $4$-cycles.]
While $d_t(u,c)$ is not concentrated enough when the graph has $4$-cycles, 
our analysis will show that the average of 
$d_{t+1}(u,c)$ over all colors in the palette of a vertex $u$
is concentrated enough.
How does this benefit us?
Markov's famous inequality may be interpreted as:
a list of $s$ positive number which average $d$
has at most $s/q$ numbers larger than $qd$
for any positive number $q$.
We modify the algorithm so that at the end of each round $t$,
every vertex $u$ removes from its palette every color $c$ with
$d_{t+1}(u,c)$ larger than $q d_{t+1}$ for some constant $q$ larger than $1$.
Look at what happens in round $t=1$.
By a straightforward application of Markov's inequality, 
instead of equation \eqref{eq:inv_sk} we will have the less stringent property:
$\forall u \in V(G_t), \forall c \in S_t(u)$ \\
\begin{align}
\label{eq:inv2_sk}
\eta_t(u) \leq \eta_t(1+o(1)), \;\; s_t(u) \geq \frac {q-1}{q} s_t(1-o(1)), 
		\;\; d_t(u,c) \leq q d_t(1+o(1)).
\end{align}
with probability $1-o(1)$.
In fact, using a generalization of Markov's inequality,
the analysis will show that with a few more modifications our algorithm 
maintains a slightly stronger property(still weaker than equation \eqref{eq:inv_sk}).

{\ }\;\;\;\; Equation \eqref{eq:inv_sk} implies that the $\eta_t(u)$, $s_t(u)$, and $d_t(u)$
at all uncolored vertices $u$ are about the same.
It is a strong statement and helps in the proofs,
but is too much to maintain on graphs with $4$-cycles.
Equation \eqref{eq:inv2_sk} is weaker and is guaranteed by our algorithm and what is more,
it is sufficient to guarantee that all uncolored vertices at the end of the second stage
have palettes with more colors than the number of uncolored adjacent vertices.
This is a key idea in our algorithm.
\item[Second Modification: Using independent random variables for easier analysis.]
Instead of waking up a vertex with some probability,
and then choosing a color from its palette uniformly at random;
for each uncolored vertex $u$ and color $c$ in its palette,
we will assign $c$ to $u$ independently with some probability.
In case multiple colors remain assigned to the vertex after the conflict resolution phase,
we will arbitrarily choose one of them to permanently color the vertex.
This modification, adapted from Johansson \cite{J96},
will make concentration of our random variables simpler.
\end{description}
Next we provide a formal description of the algorithm we have just motivated.

\subsection{A Formal Description of the Algorithm}
\label{sec:algorithm_details}
In each round of the algorithm,
some vertices are colored.
The details of the coloring procedure vary depending on which of three stages the algorithm is in.
\begin{description}
\item[First Stage]
Let $q$ be a constant greater than $1$, 
and let $(\eta_t)$, $(d_t)$, and $(s_t)$ be sequences defined recursively as
\begin{align*}
\eta_0 	&:= \Delta		& \mbox{\;\;\;}	
		&	\eta_{t+1}:= \eta_t (1- \frac{q-1}{2q^3} e^{-1/q} \frac {s_t}{d_t}) \\
d_0 		&:= \Delta		&			
		&	d_{t+1}:= d_t (1- \frac{q-1}{2q^3} e^{-1/q} \frac {s_t}{d_t}) e^{-1/q} \\
s_0	 	&:= \Delta/k	&
		&	s_{t+1} := s_t e^{-1/q}. \\
\end{align*}
\begin{equation}
\label{eq:rec_stage_1}
{\ }
\end{equation}
For round $t$, vertex $u$, and color $c$,
\begin{align}
\label{eq:pr_u_dr_c}
	\mathcal{F}_t(u,c) &:= \{c \mbox{ is not assigned to any vertex adjacent to }u\mbox{ in round }t\}
\end{align}
is an event in the probability space generated by the random choices of the algorithm in round $t$,
given the state of all data structures at the beginning of the round.

Let 
\begin{align*}
Desired\_\mathcal{F}_t &:= e^{-1/q}.
\end{align*}
%\newpage
\begin{center}
\framebox{
\begin{minipage}{15cm}

{\bf Repeat at every round $t$, until $d_{t} / s_t < 1 / q^2$}\\
{\it Phase I - Coloring Attempt }
\\
For each vertex $u$ in $G_t$, and color $c$ in $S_t(u)$: \\
\hspace*{1cm}
	Assign $c$ to $u$ with probability $\frac 1{q^2} \frac 1 {d_t}.$ \\
\\
{\it Phase II - Conflict Resolution}
\\
For each vertex $u$ in $G_t$: \\
\hspace*{.5cm}
	{\it Phase II.1} \\
\hspace*{.5cm}
	Remove from $S_t(u)$, all colors assigned to adjacent vertices. \\
\hspace*{.5cm}
	{\it Phase II.2} \\
\hspace*{.5cm}
	For each color $c$ in $S_{t}(u)$, remove $c$ from $S_{t}(u)$ with probability
\begin{equation*}
1 - min(1, \frac{Desired\_\mathcal{F}_t(u,c)}{Pr(\mathcal{F}_t(u,c)}).
\end{equation*}
\hspace*{.5cm}
	If $S_t(u)$ has at least one color which is assigned to $u$, \\
\hspace*{1cm}
		then arbirarily pick an assigned color from $S_t(u)$ to permanently color $u$. \\
\\
{\it Phase III - Cleanup(discard all colors $c$ with $d_{t+1}(u,c) \gtrsim q d_{t+1}$ from palette)}
\\
For each vertex $u$ in $G_t$: \\
\hspace*{.5cm}
	$S_{t+1}(u) = S_t(u)$.\\
\hspace*{.5cm}
	Let $\alpha = 1 - |S_{t+1}(u)| / s_{t+1}.$ \\
\hspace*{.5cm}
	If $\alpha < 0$, then $\alpha = 0$, otherwise if $\alpha > 1/q$, then $\alpha = 1/q$. \\
\hspace*{.5cm}
	Let $\gamma$ be the smallest number in $[1, \infty)$ so that
\begin{equation*}
	Average_{c \in S_{t+1}(u)} d_{t+1}(u,c) \leq \frac {1 - q \alpha}{1 - \alpha} \gamma d_{t+1}.
\end{equation*}
\hspace*{.5cm}
	Remove all colors $c$ with $d_{t+1}(u,c) \geq q \gamma d_{t+1}$ from $S_{t+1}(u)$. \\
{\bf end repeat}

\end{minipage}}
\end{center}

\item[Second Stage]
We change the recursive equations defining the constants $\eta_t$, $d_t$, and $s_t$.
\begin{align*}
\eta_{t+1} &:= \eta_t (1- \frac{q-1}{2q^3} e^{-\frac 1{q} \frac {d_t}{s_t}}) \\
d_{t+1} &:= d_t (1- \frac{q-1}{2q^3} e^{- \frac 1{q} \frac {d_t}{s_t}})e^{-\frac 1{q} \frac {d_t}{s_t}} \\
s_{t+1} &:= s_t e^{-\frac 1{q} \frac {d_t}{s_t}}
\end{align*}
\begin{equation}
\label{eq:rec_stage_2}
{\ }
\end{equation}
Also
\begin{align*}
Desired\_\mathcal{F}_t &:= e^{-\frac 1{q} \frac {d_t}{s_t}}.
\end{align*}
All other details of the repeat-until loop of the first stage are the same for the second stage,
except that an uncolored vertex $u$ is assigned a color $c$ from its palette with probability
$$\frac 1 {q^2} \frac 1 {s_t},$$
and we repeat until
$$\eta_t < \frac{q-1}{q} s_t.$$
\item[Third Stage(Greedy Coloring)]
Color each uncolored vertex $u$,
with an arbitrary color from its palette which has not been used to color an adjacent vertex.
\end{description}

\section{The Main Theorem}
\label{sec:outline}

We say that a sequence $x(n)$ is $O(f(n))$ if there is a positive number $M$ 
such that $|x(n)| \leq M |f(n)|$.
All sequences in the big-oh are indexed by $n$, the number of vertices in graph $G$.
Remember that each occurrence of the big-oh comes with a distinct constant 
$M$ which may depend on the constant $\epsilon$.
%Details can be found in Knuth \cite{K97}.

%\begin{remark}
%We ignore exact values of certain constants by resorting to the \emph{big-oh} notation.
%While details can be found in Knuth \cite{K97}, 
%the following definition will suffice for our presentation:
%we say a sequence $\{x_n\} = O(f(n))$ if $\exists M > 0$ such that $|x_n| \leq M |f(n)|$.
%\end{remark}

\begin{theorem}[Main Theorem]
Given a triangle-free graph $G$ on $n$ vertices with maximum degree 
$\Delta \geq \log^{1+\epsilon} n$ for a positive constant $\epsilon$,
and $q=7$,
there is a positive constant $c_\epsilon$ such that starting with 
$c_{\epsilon} \Delta / \log \Delta$ colors,
our algorithm finds a proper coloring of the graph in 
$O(n \Delta^2 \log \Delta)$ time with probability $1-O(1/n)$.
\end{theorem}
We need some lemmas to prove the Main Theorem and 
before that we need the following definition.
\begin{definition}
$d_{t}(v) = Average_{c \in S_{t}(v)} d_{t}(v,c)$\\
\end{definition}

\begin{lemma}[Main Lemma]
\label{lem:phase1}
Given a triangle-free graph $G$ on $n$ vertices with maximum degree 
$\Delta \geq \log^{1+\epsilon} n$ for a positive constant $\epsilon$, $q \geq 2$,
and $\psi > 1$, 
there are positive constants $\alpha_1$ and $\alpha_2$
such that for the sequences $(e_t)$ defined by
\begin{equation}
\label{eq:def_rece}
e_0 = 0, \mbox{\:\:\:} 
	e_{t+1} = \alpha_1(e_{t} + \sqrt{\frac{\psi}{d_{t}}}+ \sqrt{\frac{\psi}{s_{t}}}) \mbox{\;\; for } t > 0,
\end{equation}
and $(f_t)$ defined by
\begin{equation}
\label{eq:def_recf}
f_{t} = 1 - \alpha_2 t n^2 e^{-\psi},
\end{equation}
if $s_t \gg \psi$ and $e_t \ll 1$ at round $t$, then
$$\forall u \in V(G_t), \exists \alpha \in [0, 1/q], \forall c \in S_{t}(u),$$
\begin{align*}
s_t(u) &\geq (1-\alpha) s_{t}(1 - e_{t}) \\
d_{t}(u) &\leq \frac{1-q \alpha}{1- \alpha} d_{t}(1+ e_{t}) \\
d_{t}(u,c) &\leq q d_{t} (1+e_{t}) \\
\eta_{t}(u) &\leq \eta_{t}(1 + e_{t})
\end{align*}
with probability greater than $f_t$.
\end{lemma}
We will prove the Main Lemma in Section \ref{sec:details} and assume it in this section.
Using it we can immediately conclude the following.
\begin{corollary}
\label{cor:main_lemma}
Given the setup of the Main Lemma(Lemma \ref{lem:phase1}),
if $s_t \gg \psi$ and $e_t \ll 1$ at round $t$, then $\forall u \in V(G_t)$,
\begin{align*}
s_t(u) &\geq \frac {q-1}q s_t(1-e_t) \\
d_{t}(u) &\leq d_{t}(1+ e_{t})
\end{align*}
with probability greater than $f_t$.
\end{corollary}

\begin{lemma}
\label{lem:t_stage_1}
The first stage finishes in $O(k)$ rounds.
\end{lemma}
\begin{proof}
By the definition of sequences $(d_t)$ and $(s_t)$ in equation \eqref{eq:rec_stage_1}, we have
\begin{align*}
\frac{d_{t+1}}{s_{t+1}} &= \frac{d_{t}}{s_{t}} ( 1 - \frac {q-1} {2q^3} \frac {s_{t}}{d_{t}} e^{-1/q}) \\
				&= \frac{d_{t}}{s_{t}} - \frac {q-1} {2q^3} e^{-1/q}.
\end{align*}
Since $\frac {d_{0}}{s_{0}} = k$ we get $\frac {d_{t_1}}{s_{t_1}} \leq \frac 1 {q^2}$ for some round $t_1 = O(k)$.
\qed
\end{proof}
Let
$$t_1 = O(k)$$
be the last round of the first stage. 
Then the following lemma is a straightforward application of equation \eqref{eq:rec_stage_1}.
\begin{lemma}
\label{lem:s_t1}
\begin{equation*}
s_{t_1} = \frac {\Delta} k e^{-O(k)}
\mbox{   and   }  
\exists c > 0 \mbox{ such that if } k \leq c \log \Delta, \mbox{ then } s_{t_1} \gg 1.
\end{equation*}
\end{lemma}

\subsection{The Second Stage: Controlling the ratio of available colors to uncolored neighbors.}

We must show that $\eta_t$
decreases significantly faster than $s_t$ in the second stage.
Let
$$ \rho = 1 - \frac {q-2}{2q^3}.$$

\begin{lemma}
\label{lem:n_stage_2}
$$\eta_{t_1 + t} < \eta_{t_1} \rho^t$$
\end{lemma}
\begin{proof}
By the definition of sequence $(\eta_t)$ in equation \eqref{eq:rec_stage_2}, we have
\begin{align*}
\eta_{t+1}	&= \eta_{t} (1-\frac {q-1} {2q^3} e^{-\frac 1 {q} \frac {d_{t}}{s_{t}}}) & \\
		&< \eta_{t} (1 - \frac {q-1} {2q^3}(1 - \frac 1 {q} \frac{d_t}{s_t}))  
			& \langle e^x \geq 1 + x \rangle \\
		&< \eta_{t} (1 - \frac {q-1} {2q^3}(1 - \frac 1 {q}))  
			& \langle \frac{d_{t}}{s_{t}} < 1 \rangle \\
		&< \eta_{t} (1 - \frac {q-1} {2q^3} + \frac 1 {2q^3}) \\
		&= \eta_{t} (1 - \frac {q-2} {2q^3}) \\
		&= \eta_{t} \rho.
\end{align*}
Thus $\eta_{t_1+t} < \eta_{t_1} \rho^{t}.$
\qed
\end{proof}
We now study the ratio $d_t  / s_t$, which appears in the recursive equation 
\eqref{eq:rec_stage_2} for the sequences $(\eta_t)$, $(d_t)$, and $(s_t)$.
\begin{lemma}
\label{lem:ratio_stage_2}
$$\frac{d_{t_1+t}}{s_{t_1+t}} < \frac 1{q^2} \rho^{t}$$
\end{lemma}
\begin{proof}
By the definition of sequences $(d_t)$ and $(s_t)$ in equation \eqref{eq:rec_stage_2}, we have
\begin{equation*}
\frac{d_{t+1}}{s_{t+1}} = \frac{d_{t}}{s_{t}} (1-\frac {q-1} {2q^3} e^{-\frac 1 {q} \frac {d_{t}}{s_{t}}}).
\end{equation*}
Since the ratio of $d_{t+1}/s_{t+1}$ to $d_{t}/s_{t}$
is the same as the ratio of $\eta_{t+1}$ to $\eta_t$,
we use Lemma \ref{lem:n_stage_2} to conclude that
\begin{equation*}
\frac{d_{t_1+t}}{s_{t_1+t}} < \frac{d_{t_1}}{s_{t_1}} \rho^t < \frac 1{q^2} \rho^{t}.
\end{equation*}
\qed
\end{proof}

\begin{lemma}
\label{lem:s_stage_2}
If $q$ is greater than $2$, then
$$s_{t_1+t} \geq s_{t_1} (1 - \frac {4} {q-2}).$$
\end{lemma}
\begin{proof}
By the definition of sequence $(s_t)$ in equation \eqref{eq:rec_stage_2}, we have
\begin{align*}
s_{t+1} &= s_{t} e^{-\frac 1 {q} \frac {d_{t}}{s_{t}}} \\
	&\geq s_{t} e^{- \frac 1 {q} \frac {d_{t}}{s_{t}}} \\
	&\geq s_{t} (1 - \frac 1 {q} \frac {d_{t}}{s_{t}}).
\end{align*}
Thus
\begin{align*}
s_{t_1+t} &\geq s_{t_1} \prod_{i=0}^{t} (1 - \frac 1 {q} \frac {d_{t}}{s_{t}}) & \\
	&\geq s_{t_1} \prod_{i=0}^{t} (1- \frac 1 {q^3} \rho^i) 
		& \langle \mbox{Lemma \ref{lem:ratio_stage_2}}\rangle \\
	&\geq s_{t_1} exp(- \frac 2 {q^3} \sum_{i=0}^{t} \rho^i) 
		& \langle \mbox{since } q > 2 \rangle\\
	&\geq s_{t_1} exp(- \frac 2 {q^3} \frac 1 {1-\rho})  
		& \langle \mbox{sum of geometric series}\rangle \\
	&= s_{t_1} exp(- \frac {4q^3} {q^3(q-2)}) & \\
	&= s_{t_1} exp(- \frac {4} {q-2}) & \\
	&\geq s_{t_1} (1 - \frac {4} {q-2}). 
\end{align*}
\qed
\end{proof}

\begin{lemma}
\label{lem:t_stage_2}
For any $\delta > 0$, there is a $c_{\delta} > 0$ such that if $q > 6$,
$k \leq c_{\delta} \log \Delta$, and $\Delta \gg 1$,
then the second stage takes at most $\delta \log \Delta$ rounds.
\end{lemma}
\begin{proof}
By Lemma \ref{lem:s_t1}, we have
$$s_{t_1} = \frac{\Delta} k e^{-O(k)}.$$
Since $q > 6$, the above equation and Lemma \ref{lem:s_stage_2} imply that 
$$s_{t_1+t} = \Omega(\frac{\Delta} k e^{-O(k)})$$ 
for all $t$.
Also, by Lemma \ref{lem:n_stage_2}, we have
$$\eta_{t_1+t} < \Delta \rho^t.$$
Since $\Delta \gg 1$, given a $\delta$ it is straightforward to find a $c_{\delta} > 0$
so that if $k \leq c_{\delta} \log \Delta$,
then after $\delta \log \Delta$ rounds $\eta_t < \frac{q-1}q s_t.$ 
\qed
\end{proof}

\subsection{Bounding the Error Estimate in all Concentration Inequalities}
Now we look at $d_{t}$, which is used to bound the error term $e_{t}$. Let
$$\mu = 1 - \frac 1 {2q^2}.$$
\begin{lemma}
\label{lem:d_stage_2}
$$d_{t_1+t} \geq d_{t_1} (1 - \frac {4}{q-2}) \mu^{t}$$
\end{lemma}
\begin{proof}
By the definition of sequence $(d_t)$ in equation \eqref{eq:rec_stage_2}, we have
\begin{align*}
d_{t+1} &= d_{t}(1 - \frac {q-1} {2q^3} e^{-\frac 1 {q} \frac {d_{t}}{s_{t}}})e^{-\frac 1 {q} \frac {d_{t}}{s_{t}}} \\
	&\geq d_{t}(1 - \frac {1} {2q^2})e^{-\frac 1 {q} \frac {d_{t}}{s_{t}}} \\
	&= d_{t} \mu \; e^{-\frac 1 {q} \frac {d_{t}}{s_{t}}}.
\end{align*}
Combining the inequality above with the definition of sequence $(s_t)$ in equation \eqref{eq:rec_stage_2}, we get
$$\frac {d_{t+1}}{d_t} \geq \mu \frac {s_{t+1}}{s_t}.$$ 
We now use Lemma \ref{lem:s_stage_2} to conclude that
$$ \frac {d_{t_1+t}}{d_{t_1}} \geq \mu^t \frac {s_{t_1+t}}{s_{t_1}} 
	\geq \mu^t (1 - \frac {4}{q-2}).$$
\qed
\end{proof}
Let $t_2$ be the number of rounds spent in the second stage.
\begin{lemma}
\label{lem:e_stage_1_2}
There is a positive constant $\alpha$ such that in the first two stages of our algorithm,
$$e_{t} \leq \alpha^{t} O(\sqrt{\frac {k \; e^{O(k)} \psi}{\Delta \mu^{t_2}}}).$$
\end{lemma}
\begin{proof}
By Lemma \ref{lem:s_t1}, we have
$$s_{t_1} = \frac{\Delta} k e^{-O(k)}.$$
Note that in equation \eqref{eq:def_rece}, the recurrence for $e_{t}$,
the largest term is $O(\sqrt{\psi / s_t})$ in the first stage, 
while $O(\sqrt{\psi / d_t})$ is larger in the second.
At round $t_1$, the algorithm moves to the second stage and
$d_{t_1} = \Theta (s_{t_1})$.
The greedy stage starts at round $t_1 + t_2$. 
Since both sequences $(d_t)$ and $(s_t)$ are decreasing,
we use Lemma \ref{lem:d_stage_2} to conclude that
$$O(\sqrt{\frac{\psi}{d_{t_1+t_2}}}) = O(\sqrt{\frac {\psi}{s_{t_1} \mu^{t_2}}}) = 
O(\sqrt{\frac {k \; e^{O(k)} \psi}{\Delta \mu^{t_2}}})$$ 
is the maximum this error term can be. Thus we can simplify the recurrence for $e_{t}$ to
$$e_{t+1} = O(e_{t} + \sqrt{\frac {k \; e^{O(k)} \psi}{\Delta \mu^{t_2}}}).$$
Since $e_{0} = 0$, a simple upper bound for $e_t$ is given by
$$e_{t} \leq \alpha^{t} O(\sqrt{\frac {k \; e^{O(k)} \psi}{\Delta \mu^{t_2}}})$$
where $\alpha$ is some positive constant.
\qed
\end{proof}

\begin{lemma}
\label{lem:pre_main}
Given a triangle-free graph $G$ on $n$
vertices with maximum degree $\Delta \geq \log^{1+\epsilon} n$ for a positive constant $\epsilon$,
and $q = 7$,
there is a positive constant $c_{\epsilon}$ such that,
with $\Delta / k$ colors where $k \leq c_{\epsilon} \log \Delta$,
our algorithm reaches the greedy stage at round 
$\tau = O(\log \Delta)$ with $e_{\tau} \ll 1$, and
$\forall u \in V(G_{\tau})$
\begin{align*}
s_{\tau}(u) &\geq \frac {q-1} q s_{\tau}(1 - e_{\tau}) \\
\eta_{\tau}(u) &\leq \eta_{\tau}(1 + e_{\tau})
\end{align*}
with probability $1-O(1/n)$.
\end{lemma}
\begin{proof}
Let $\psi = 3 \log n$,
and let $\tau = t_1 + t_2$ be the number of rounds to reach the greedy stage. 
Using Lemma \ref{lem:e_stage_1_2}, we get
$$ \exists \alpha, \beta > 0, \;\;\; 
	e_{\tau} \leq \alpha^{t_1} \beta^{t_2} O(\sqrt{\frac{k e^{O(k)}\psi}{\Delta}}).$$
Since $\Delta \geq \log^{1+\epsilon} n$,
it is straightforward to show that
$$\exists c_1 > 0, \;\;\; e_{\tau} \ll 1$$
if $t_1, t_2, k \leq c_1 \log \Delta$.
Since $t_1 = O(k)$, by Lemma \ref{lem:t_stage_2}
$$\exists c_2 > 0, \;\;\; t_1, t_2 \leq c_1 \log \Delta, \;\;\; \mbox{and} \;\;\; s_{\tau} \gg \psi$$
if $k \leq c_2 \log \Delta$.

Let $c_{\epsilon} = min(c_1,c_2)$.
The above computations show that if $k \leq c_{\epsilon} \log \Delta$,
then  $\tau = t_1 + t_2 = O(k + \log \Delta) = O(\log \Delta)$ and $e_{\tau} \ll 1$.
Applying Corollary \ref{cor:main_lemma} completes the proof.
\qed
\end{proof}
We may now prove the Main Theorem.
\begin{proof}[of the Main Theorem]
Using Lemma \ref{lem:pre_main}, we only need to compute the time complexity of our algorithm.
The number of steps taken in the greedy coloring stage
is dominated by that of the first and second stages.
In each round of these stages, 
we compute the probability of the event $\mathcal{F}_t(u,c)$ 
defined in equation \eqref{eq:pr_u_dr_c}.
This requires iterating through each uncolored vertex $u$ and color $c$ in its palette, 
and examining all adjacent vertices.

Thus each round takes $O(n \Delta^2)$ steps and by Lemma \ref{lem:pre_main},
the algorithm requires $\tau = O(\log \Delta)$ rounds to reach the greedy coloring stage.
With probability $1-O(1/n)$, we get
$$\forall u \in V(G_{\tau}), \;\;\; \eta_{t}(u) < s_{\tau}(u),$$
which implies that the greedy coloring stage will successfully complete the coloring.
\qed
\end{proof}

\section{Several Useful Inequalities}
\label{sec:prelim}
Now we look at some preliminaries which will be used in the proof details.
The next lemma describes what happens to the average value of a finite subset of
real numbers when large elements are removed. 
As shown in the statement of the lemma, 
it implies Markov's Inequality \cite{AS08}.
\begin{lemma}
\label{lem:mu1}
Consider a set of positive real numbers of size $n$ and average value $\mu$.
If we remove $\alpha n$ elements with value atleast $q \mu$ for some $q > 1$,
then the remaining points have average
$$ \mu' \leq \mu \frac{1-q \alpha}{1- \alpha}.$$
In particular, $\alpha \leq \frac 1 q$ since $\mu' \geq 0$.
\end{lemma}
\begin{proof}
The conclusion is obtained by a trivial manipulation of the following inequality
which relates $\mu$ and $\mu'$.
$$ q \mu \alpha + \mu' (1 - \alpha ) \leq \mu $$
\qed
\end{proof}
The next lemma describes what  happens when we add large elements to a finite
subset of real numbers.
\begin{lemma}
\label{lem:mu2}
Given the setup of Lemma \ref{lem:mu1}, 
if we add $\alpha n$ points with value $q \mu$ to the sample,
then the resulting larger sample has average
$$ \mu' = \mu \frac {1+q \alpha}{1+\alpha}$$
\end{lemma}
\begin{proof}
The conclusion is easily obtained from the following equation relating $\mu$ and $\mu'$.
$$ \mu'(1+\alpha) = \mu + q \mu \alpha $$
\qed
\end{proof}
We use the following lemma for computations with error factors.
\begin{lemma}
\label{lem:O}
Let $(A_n)$ be a sequence such that $0 < A_n < c < 1$(where $c$ is a constant),
and let $(e_n)$ be another sequence. Then
$$ 1 - A_n(1 + e_n) = (1 - A_n)(1+O(e_n)) $$
\end{lemma}
\begin{proof}
\begin{eqnarray*}
1 - A_n(1+e_n) &=& (1-A_n)(1+e_n) - e_n \\
&=& (1-A_n)(1+e_n) - (1-A_n) \frac {e_n}{(1-A_n)} \\
&=& (1-A_n)(1+e_n) - (1-A_n) O(e_n) \\
&=& (1-A_n)(1+O(e_n))
\end{eqnarray*}
\qed
\end{proof}
We use the following version of Azuma's inequality \cite{MR01}
to prove concentration of random variables.
\begin{alphatheorem}[Azuma's inequality]
\label{thm:azuma}
Let $X$ be a random variable determined by $n$ trials
$T_1, \dots, T_n$, such that for each $i$,
and any two possible sequences of outcomes
$t_1, \dots, t_i$ and $t_1, \dots, t_{i-1}, t_i'$:
$$|E[X | T_1=t_1, \dots, T_i = t_i] - E[X | T_1=t_1, \dots, T_i = t_i']| \leq \alpha_i$$
then
$$Pr(|X - E[X]| > t) \leq 2 e^{-t^2/(\sum \alpha_i^2)}$$
\end{alphatheorem}

\section{Proof of the Main Lemma}
\label{sec:details}

\begin{proof}[of the Main Lemma]
The proof is by induction on the round number $t$ using the lemmas that follow.
The base case, when $t=0$, is trivially true.
Lemmas \ref{lem:A}, \ref{lem:S}, and \ref{lem:D} give us the induction step for the 
first stage; Lemmas  \ref{lem:2A}, \ref{lem:2S}, and \ref{lem:2D} give
the induction step for the second. 
\qed
\end{proof}

All events below are in the probability space generated by our randomized algorithm
in round $t+1$,
given the state of all data structures at the beginning of the round.
The following assumptions are the induction hypothesis and are repeatedly used in all
the lemmas of this section.

\begin{assumption}
\label{ass:indhyp}
Assume
$$q \geq 2, s_t \gg \psi, e_t \ll 1$$
and
$$\forall u \in V(G_t), \forall c \in S_{t}(u), \exists \alpha \in [0, 1/q],$$
\begin{align*}
s_t(u) &\geq (1-\alpha) s_{t}(1 - e_{t}) \\
d_{t}(u) &\leq \frac{1-q \alpha}{1- \alpha} d_{t}(1+ e_{t}) \\
d_{t}(u,c) &\leq q d_{t} (1+e_{t}) \\
\eta_{t}(u) &\leq \eta_{t}(1 + e_{t}).
\end{align*}
\end{assumption}

%\subsection{A glossary of variables used in the proof details}

%\begin{tabular}{|l|p{10cm}|} 
%\hline
%$\tilde S_v$ & 
%\begin{tabular}{ll} 
%$S_v -$ & \{colors permanently taken by neighbors of $v$\} $\bigcup$ \\
%	     & \{colors dropped by $v$ at phase III\}
%\end{tabular}\\
%\hline
%$\tilde d_{v,c}$ &
%\begin{tabular}{lp{7cm}} 
%$d_{v,c} -$	& \{neighbors of $v$ that have dropped $c$ from their palatte by the end of phase III\} $\bigcup$ \\
%			&  \{permanently colored neighbors of $v$\}
%\end{tabular}\\
%\hline
%$Up(X)$ & The vertices in set $X$ that wake up and attempt to color themselves. \\
%\hline
%\end{tabular}

\subsection{Expected Values and Concentration Inequalities for the First Stage}
\label{sec:first_stage}
We next consider the state of the palettes just before the cleanup phase of round $t$.

\begin{definition}
Let ${\tilde S}_t(u)$ be the list of colors in the palette of vertex $u$
in round $t$ just before the cleanup phase,
and let ${\tilde s}_t(u)$ be the size of ${\tilde S}_t(u)$. 
That is, ${\tilde S}_t(u)$ is obtained from $S_t(u)$
by removing colors discarded in the conflict resolution phase.
\end{definition}

Consider the event
\begin{equation}
\label{eq:def_tilde_s}
\mathcal{\tilde S} = \{\forall u \in V(G_{t+1}), s_{t}(u) e^{-1/q}(1 - O(e_t + \sqrt{\frac{\psi}{s_{t}}} + \frac{1}{d_{t}})) \leq \tilde s_{t}(u) \leq s_{t}(u) e^{-1/q}(1 + O(e_t + \sqrt{\frac{\psi}{s_{t}}} + \frac{1}{d_{t}}))\}.
\end{equation}

\begin{lemma}
\label{lem:sg}
Given Assumption \ref{ass:indhyp}, we have
$$ Pr(\mathcal{\tilde S}) \geq 1- e^{-\psi}O(n).$$
\end{lemma}

\begin{proof}
Suppose $u$ is an uncolored vertex at the beginning of round $t$, 
and $c$ a color in its palette.
\begin{align*}
\lefteqn{Pr\{\text{$c$ is removed from $S_{t}(u)$ in phase II.1}\}}  \;\;\;\;\;\; \\
	&= 1 - Pr\{\text{no neighbor of $u$ is assigned $c$}\} \\
	&= 1 - \prod_{v \in D_{t}(u,c)}(1 - Pr\{ \text{$v$ is assigned $c$} \}) \\
	&= 1 - \prod_{v \in D_{u,c}}(1 - \frac 1 {q^2} \frac 1 {d_{t}} ) \\
	&\leq 1 - (1 - \frac 1 {q^2} \frac 1 {d_{t}})^{d_{t}(u,c)} \\
	&\leq 1 - (1- \frac 1 {q^2} \frac 1 {d_{t}})^{q d_{t}(1+e_{t})} \\
	&\leq 1 - e^{\log (1 - \frac 1 {q^2} \frac 1 {d_{t}})q d_{t}(1+e_{t})} 
		& \langle \log(1+x) = x + O(x^2) \rangle \\
	&\leq 1 - e^{(-\frac 1 {q^2} \frac 1 {d_{t}} + O(\frac 1 {d_{t}})^2) q d_{t}(1+e_{t})} 
		& \langle \mbox{Assumption \ref{ass:indhyp}} \rangle \\
	&\leq 1 - e^{-1/q}(1+O(e_{t} + \frac{1}{d_{t}}))
\end{align*}
In phase II.2 of round $t+1$ we remove colors from the palette using an appropriate bernoulli variable, to get
\begin{equation}
\label{eq:u_dr_c_naturally}
Pr\{c \notin \tilde S_{t}(u)\} = 1 - e^{-1/q}(1 + O(e_t + \frac {1}{d_{t}})).
\end{equation}
Using linearity of expectation
$$ \forall u \in V(G_{t+1}), E[\tilde s_{t}(u)] = s_{t}(u) e^{-1/q}(1+O(e_t + \frac{1}{d_{t}})).$$

For concentration of $\tilde s_t(u)$,
suppose $s_t(u) = m$.
Let $c_1, \dots, c_m$ be the colors in $S_t(u)$.
Then $\tilde S_t(u)$ may be considered a random variable determined by $m$
trials $T_1, \dots, T_m$ where $T_i$ is the set of vertices in $G_t$
that are assigned color $c_i$ in round $t$. Observe that $T_i$
affects $\tilde S _t(u)$ by at most 1 given $T_1, \dots, T_{i-1}$.
Now using Theorem \ref{thm:azuma} we get,
$$Pr\{|\tilde s_t(u) - E[\tilde s_t(u)]| \geq \sqrt{\psi s_t(u)}\} \leq e^{-\psi} O(1).$$
We end the proof using the union bound for probabilities.
\qed
\end{proof}
We now focus on the sets $D_t(u,c)$. The following two lemmas will help.
\begin{lemma}
\label{lem:u_cl_c}
Let $u$ be an uncolored vertex, and $c$ a color in its palette at the beginning of round $t$.
Then given Assumption \ref{ass:indhyp}, we have
$$Pr\{\text{$u$ is assigned  $c$ and $c \in \tilde S_t(u)$}\} = 
	\frac 1 {q^2} \frac 1 {d_{t}} e^{-1/q}(1+O(e_{t} + \frac {1}{d_{t}})).$$
\end{lemma}
\begin{proof}
\begin{align*}
	\lefteqn{Pr\{ \text{$u$ is assigned  $c$ and $c \in \tilde S_t(u)$} \}} \;\;\;\;\;\; \\ 
		&= Pr\{ \text{$u$ is assigned $c$}\} Pr\{ c \in \tilde S_t(u)\} \\
		&= \frac 1 {q^2} \frac 1 {d_{t}} e^{-1/q}(1+O(e_{t} + \frac{1}{d_{t}}))
			& \langle \mbox{Equation \eqref{eq:u_dr_c_naturally}} \rangle
\end{align*}
\qed
\end{proof}
The following lemma is a consequence of the previous one.
\begin{lemma}
\label{lem:u_cl}
Let $u$ be an uncolored vertex at the beginning of round $t$. 
Then given Assumption \ref{ass:indhyp}, we have
$$Pr\{u\text{ is colored}\} 
	\geq \frac {q-1}{2 q^3} \frac {s_t} {d_{t}} e^{-1/q}(1+O(e_{t} + \frac {1}{d_{t}})).$$
\end{lemma}
\begin{proof}
Consider the event
$$ \{u \mbox{ is colored}\} = 
	\bigcup_{c \in S_t(u)} \{u \mbox{ is assigned } c \mbox{ and } c \in \tilde S_t(u)\}.$$
Since the events in the union on the right hand side of the equation above are independent,
$$Pr\{u \mbox{ is colored}\} = 1 - 
	\prod_{c \in S_t(u)}
		(1 - Pr\{u \mbox{ is assigned } c \mbox{ and } c \in \tilde S_t(u)\}).$$
Now using Lemma \ref{lem:u_cl_c}, we get
\begin{align*}
\lefteqn{Pr \{u \mbox{ is colored}\}} \;\;\;\;\;\; \\ 
	&\geq 1 - (1 - \frac 1 {q^2} \frac 1 {d_t} e^{-1/q}(1 + O(e_t + \frac 1 {d_t})))^{s_t(u)}\\
	&\geq 1 - (1 - \frac 1 {q^2} \frac 1 {d_t} e^{-1/q}
		(1 + O(e_t + \frac 1 {d_t})))^{\frac {q-1}q s_t(1-e_t)} 
		& \langle \mbox{Assumption \ref{ass:indhyp}} \rangle \\
	&\geq 1 - exp(- \frac {q-1} {q^3} \frac {s_t} {d_t} e^{-1/q}(1 + O(e_t + \frac 1 {d_t}))) \\
	&\geq 1 - (1 - \frac {q-1} {2q^3} \frac {s_t} {d_t} e^{-1/q}(1 + O(e_t + \frac 1 {d_t}))) 
		& \langle \mbox{since $q \geq 2$} \rangle \\
	&=  \frac {q-1} {2q^3} \frac {s_t} {d_t} e^{-1/q}(1 + O(e_t + \frac 1 {d_t})).
\end{align*}
\qed
\end{proof}
\begin{samepage}
We will need the following definitions.
\begin{definition} {\ } \\
\begin{itemize}
\item
Let $\tilde D_t(u,c)$ be the set of uncolored vertices that have color $c$
in their palettes  and are uncolored in round $t$, just before the cleanup phase. That is,
$$\tilde D_t(u,c) = D_t(u,c) \setminus (\{v | c \notin \tilde S _t(v)\} \cup \{v | v \mbox{ is colored in round } t\}).$$
\item
Let $\tilde d_t(u,c)$ be the size of $\tilde D_t(u,c)$.
\item
$ \bar d_{t}(u) := \sum_{c \in \tilde S_{t}(u)} \tilde d_{t}(u,c) = \sum_{c \in S_{t}(u)} 1_{\{c \in \tilde S_{t}(u)\}} \tilde d_{t}(u,c) $
\item
$ \tilde d_{t}(u) := \frac{\bar d_{t}(u)}{\tilde s_{t}(u)}$
\end{itemize}
\end{definition}
\end{samepage}
Now consider the event
\begin{equation}
\label{eq:def_tilde_a}
\mathcal{\tilde A} := \{\forall u \in V(G_{t+1}), \tilde d_{t}(u) 
\leq d_{t}(u) (1 - \frac {q-1}{2q^3} \frac {s_{t}}{d_{t}} e^{-1/q})e^{-1/q}(1 + O(e_{t} + \sqrt{\frac{\psi} {s_{t}}} + \frac{1}{d_{t}} + \sqrt{\frac{\psi d_{t}}{s_{t} d_{t}(u)}}))\}.
\end{equation}

\begin{lemma}
\label{lem:tilde_a}
Given Assumption \ref{ass:indhyp}, we have
$$Pr(\mathcal{\tilde A}) \geq 1 - e^{-\psi}O(n^2).$$
\end{lemma}
\begin{proof}
Let $u$ be an uncolored vertex at the beginning of round $t$,
and let $c$ be a color in its palette.
For a vertex $v$ in $D_{t}(u,c)$, 
Lemma \ref{lem:u_cl_c} implies that $Pr\{v$ is colored with $d\} = O(1 / {d_t})$
for any color $d$ in $S_t(v)$.
Thus,
\begin{eqnarray*}
Pr(\{c \notin \tilde S_t(v) \} \cap \{v \mbox{ is colored}\}) 
	&=& \sum_{d \in S_t(v)} Pr(\{c \notin \tilde S_t(v) \} \cap \{v \mbox{ is colored with } d\}) \\
	&=& \sum_{d \in S_t(v)} Pr\{c \notin \tilde S_t(v) | v \mbox{ is colored with } d \} 
			Pr\{v \mbox{ is colored with } d\} \\
	&=& Pr\{c \notin \tilde S_t(v)\}(1 + O(\frac 1 {d_t})) \sum_{d \in S_t(v)}  
			Pr\{v \mbox{ is colored with } d\} \\
	&=& Pr\{c \notin \tilde S_t(v)\}Pr\{v \mbox{ is colored}\}(1 + O(\frac 1 {d_t})).
\end{eqnarray*}
A straightforward computation now shows that
\begin{equation}
\label{eq:v_dr_c_n_cln}
Pr(\{c \notin \tilde S_t(v) \} \cap \{v \mbox{ is not colored}\}) = \\
	Pr\{c \notin \tilde S_t(v)\}Pr\{v \mbox{ is not colored}\}(1 + O(\frac 1 {d_t})).
\end{equation}
Now, $v$ is removed from the set $D_t(u,c)$ if either it is colored or color $c$
is removed from its palette. This means that event
\begin{eqnarray*}
\{v \notin \tilde D_t(u,c) \} = \{\text{$v$ is colored}\} 
	\cup (\{c \notin \tilde S_t(v)\} \cap \{ \text{$v$ is not colored}\}).
\end{eqnarray*}
Since $G$ is triangle-free, $u$ and $v$ do not have any common neighbors.
This implies that
\begin{align*}
\lefteqn{Pr\{v \notin \tilde D_t(u,c) | c \in \tilde S_t(u)\}} \;\;\;\;\;\; \\
	&= Pr\{v \notin \tilde D_t(u,c)\}(1+O(\frac 1 {d_t})) \\
	&= (Pr\{\text{$v$ is colored}\}
		+ Pr(\{c \notin \tilde S_t(v) \} \cap \{ \text{$v$ is not colored}\}))(1+O(\frac 1 {d_t})) \\
	&= (Pr\{\text{$v$ is colored}\} 
		+ Pr\{c \notin \tilde S_t(v) \} Pr\{ \text{$v$ is not colored}\}) (1+O(\frac 1 {d_t})) 
		& \langle \mbox{equation \eqref{eq:v_dr_c_n_cln}}\rangle \\
	&= (Pr\{\text{$v$ is colored}\} 
		+ (1 - e^{-1/q})(1 - Pr\{\text{$v$ is colored}\}))(1 + O(e_t + \frac {1}{d_{t}})) 
		& \langle \mbox{equation \eqref{eq:u_dr_c_naturally}}\rangle \\		
	&= (1 - (1 - Pr\{\text{$v$ is colored}\}) e^{-1/q})(1 + O(e_t + \frac {1}{d_{t}})) \\
	&\geq (1 - (1 - \frac {q-1} {2q^3} \frac {s_{t}} {d_{t}} e^{-1/q}) e^{-1/q})
											(1+ O(e_{t} + \frac {1}{d_{t}}))
		& \langle \mbox{Lemma \ref{lem:u_cl}}\rangle.
\end{align*}
Using linearity of expectation
\begin{equation}
E[\tilde d_{t}(u,c) | c \in \tilde S_{t}(u)] 
= E[\tilde d_{t}(u,c)](1+O(\frac 1{d_t})) 
\leq d_{t}(u,c)(1 - \frac {q-1}{2q^3} \frac {s_{t}} {d_{t}} e^{-1/q}) e^{-1/q}(1+ O(e_{t} + \frac {1}{d_{t}})).
\end{equation}
Now using the above bound
\begin{eqnarray*}
E[\bar d_{t}(u)] &=& \sum_{c \in S_{t}(u)} Pr\{c \in \tilde S_{t}(u)\} E[\tilde d_{t}(u,c) | c \in \tilde S_{t}(u)] \\
	&\leq& e^{-1/q} \sum_{c \in S_{t}(u)} d_{t}(u,c)(1 - \frac {q-1}{2q^3} \frac {s_{t}}{d_{t}} e^{-1/q}) e^{-1/q}(1+ O(e_{t} + \frac {1}{d_{t}})) \\
	&\leq& e^{-1/q} s_{t}(u) d_{t}(u)(1 - \frac {q-1}{2q^3} \frac {s_{t}}{d_{t}} e^{-1/q}) e^{-1/q}(1+ O(e_{t} + \frac {1}{d_{t}}))
\end{eqnarray*}

For concentration of $\bar d_t(u)$,
suppose $s_t(u) = m$.
Let $c_1, \dots, c_m$ be the colors in $S_t(u)$.
Then $\bar d_t(u)$ may be considered a random variable determined by the random trials
$T_1, \dots, T_m$, where $T_i$ is the set of vertices in $G_t$ that are assigned color $c_i$
in round $t$.
Observe that $T_i$ affects $\bar d_t(u)$ by at most $d_t(u,c)$.

Thus $\sum \alpha_i^2$ in the statement of Theorem \ref{thm:azuma}
is less that $\sum_{c \in S_t(u)} d_t^2(u,c)$.
This upperbound is maximized when the $d_{t}(u,c)$ take the extreme values of
$q d_{t}$ and $0$ subject to $d_{t}(u) = \frac 1 {s_{t}(u)} \sum_{c \in S_{t}(u)} d_{t}(u,c)$.
Thus
$$\sum \alpha_i^2 \leq O( (d_{t})^2 d_{t}(u) s_{t}(u) / d_{t}) \leq O(s_{t}(u) d_{t} d_{t}(u))$$
Using Theorem \ref{thm:azuma}, we get
\begin{align*}
Pr\{& \bar d_{t}(u) 
- e^{-1/q} s_t(u) d_{t}(u)(1 - \frac {q-1} {2q^3} \frac {s_{t}} {d_{t}} e^{-1/q}) e^{-1/q}(1 + O(e_{t} + \frac {1}{d_{t}})) \geq O(\sqrt{\psi s_t(u)d_{t} d_{t}(u)}) \} \\
	&\;\;\;\;\;\;\;\;\;\;\;\;  \leq e^{-\psi} O(1).
 \end{align*}
Lemma \ref{lem:sg} says that the event $\mathcal{\tilde S}$ 
occurs with probability $1-e^{-\psi}O(n)$. Thus
\begin{align*}
Pr\{& \frac{\bar d_{t}(u)}{\tilde s_{t}(u)} 
- d_{t}(u)(1 - \frac {q-1} {2q^3} \frac {s_{t}} {d_{t}} e^{-1/q}) e^{-1/q}(1 + O(e_{t} + \frac {1}{d_{t}} + \sqrt{\frac{\psi}{s_{t}}})) \geq O(\sqrt{\frac{\psi d_{t} d_{t}(u)}{s_{t}}}) \} \\
	& \;\;\;\;\;\;\;\;\;\;\;\; \leq e^{-\psi} O(n).
 \end{align*}
Therefore
$$Pr\{\frac{\bar d_{t}(u)}{\tilde s_{t}(u)} 
\geq d_{t}(u)(1 - \frac {q-1} {2q^3}\frac {s_{t}}{d_{t}} e^{-1/q}) e^{-1/q}
(1 + O(e_{t} + \sqrt{\frac{\psi} {s_{t}}} + \frac{1}{d_{t}} + \sqrt{\frac{\psi d_{t}}{s_{t} d_{t}(u)}}))\} 
\leq e^{-\psi} O(n).$$
We end the proof using the union bound for probabilities.
\qed
\end{proof}

Note that $\frac{\bar d_{t}(u)}{\tilde s_{t}(u)}$ is the average $|\tilde D_t(u,c)|$ 
at a vertex $u$ at the end phase II.
Phase III only brings this average down by removing colors with large $d_{u,c}$.
Thus we get the next lemma almost immediately.
Consider the event
$$\mathcal{A} := \{\forall u \in V(G_{t+1}), d_{t+1}(u) \leq d_{t+1}(1 +O( e_{t} + \sqrt{\frac{\psi} {s_{t}}} + \frac{1}{d_{t}})\}.$$

\begin{lemma}
\label{lem:A}
Given Assumption \ref{ass:indhyp}, we have
$$Pr(\mathcal{A}) \geq 1 - e^{-\psi} O(n^2).$$
\end{lemma}
\begin{proof}
Assume the occurrence event $\mathcal{\tilde A}$, as defined in equation \eqref{eq:def_tilde_a}, and let $u$ be a vertex in $V(G_{t+1})$.
Then
\begin{align*}
\lefteqn{d_{t}(u)(1 - \frac {q-1} {2q^3} \frac {s_{t}}{d_{t}} e^{-1/q}) e^{-1/q}
 (1 + O(e_{t} + \sqrt{\frac{\psi} {s_{t}}} + \frac{1}{d_{t}} + \sqrt{\frac{\psi d_{t}}{s_{t} d_{t}(u)}})) }
 	\;\;\;\;\;\; \\
&= d_{t}(u)(1 - \frac {q-1} {2q^3} \frac {s_{t}}{d_{t}} e^{-1/q}) e^{-1/q}
(1 + O(e_{t} + \sqrt{\frac{\psi} {s_{t}}} + \frac{1}{d_{t}})) + O(\sqrt{\frac{\psi d_{t}(u) d_{t}}{s_{t}}}) \\
&= d_{t}(u)(1 - \frac {q-1} {2q^3} \frac {s_{t}}{d_{t}} e^{-1/q}) e^{-1/q}
(1 + O(e_{t} + \sqrt{\frac{\psi} {s_{t}}} + \frac{1}{d_{t}})) + d_{t} O(\sqrt{\frac{\psi d_{t}(u)}{s_{t} d_{t}}}) \\
&\leq d_{t}(1 - \frac {q-1} {2q^3} \frac {s_{t}}{d_{t}} e^{-1/q}) e^{-1/q}
(1 + O(e_{t} + \sqrt{\frac{\psi} {s_{t}}} + \frac{1}{d_{t}})) + d_{t} O(\sqrt{\frac{\psi}{s_{t}}}) \\
&= d_{t}(1 - \frac {q-1} {2q^3} \frac {s_{t}}{d_{t}} e^{-1/q}) e^{-1/q}
(1 + O(e_{t} + \sqrt{\frac{\psi} {s_{t}}} + \frac{1}{d_{t}})).
\end{align*}
Thus the event
\begin{align*}
\lefteqn{\{\frac{\bar d_{t}(u)}{\tilde s_{t}(u)} 
\geq d_{t}(1 - \frac {q-1} {2q^3} \frac {s_{t}}{d_{t}} e^{-1/q}) e^{-1/q}
(1 + O(e_{t} + \sqrt{\frac{\psi} {s_{t}}} + \frac{1}{d_{t}}))\} }
 	\;\;\;\;\;\; \\
&\subseteq \{\frac{\bar d_{t}(u)}{\tilde s_{t}(u)} 
\geq d_{t}(u)(1 - \frac {q-1} {2q^3} \frac {s_{t}}{d_{t}} e^{-1/q}) e^{-1/q}
(1 + O(e_{t} + \sqrt{\frac{\psi} {s_{t}}} + \frac{1}{d_{t}} + \sqrt{\frac{\psi d_{t}}{s_{t} d_{t}(u)}}))\}.
\end{align*}
Therefore
$$Pr(\mathcal{A}) \geq Pr(\mathcal{\tilde A}) \geq e^{-\psi} O(n^2).$$
\qed
\end{proof}

Next we show that in the cleanup phase of round $t$, 
no vertex discards so many colors that its palette size in round $t+1$ 
becomes less than
$\frac {q-1}q s_{t+1}(1-e_{t+1})$.
Consider the event
\begin{align}
\mathcal{S} := 
\{
\forall v \in V(G_{t+1}), & \exists \alpha \in [0, \frac 1 q] \text{ such that } \\
	 & s_{t+1}(u) \geq (1-\alpha) s_{t+1}(1 + O(e_{t} + \sqrt{\frac{\psi}{s_t}} + \frac{1} {d_{t}})), \\
	& d_{t+1}(u) \leq \frac{1-q \alpha}{1-\alpha} d_{t+1}(1+ O(e_{t} + \sqrt{\frac{\psi}{s_t}} + \frac{1} {d_{t}}))
	\}.
\end{align}
\begin{lemma}
\label{lem:S}
Given Assumption \ref{ass:indhyp}, we have
$$Pr(\mathcal{S}) \geq 1 - e^{-\psi} O(n^2).$$
\end{lemma}
\begin{proof}
%In this regard, we may assume that $d_{t}(u) \geq \Delta^{1-q\epsilon}$, 
%since otherwise, by Markov's inequality, we have that the maximum number of
%colors removed from $\tilde s_{t}(u)$(colors with $\tilde d_{t}(u,c) \geq q d_{t+1}$) is
%$\tilde s_{t}(u) O(\frac 1 {\Delta^{2\epsilon}})$ 
%which is already accounted for in the statement of the lemma.

Consider vertex $u \in V(G_{t+1})$. 
Using Assumption \ref{ass:indhyp}, at round $t$, $\exists \alpha \in [0, \frac 1 q]$ such that $ s_{t}(u) \geq (1-\alpha) s_{t}(1 - e_{t}) $
and $d_{t}(u) \leq \frac {1-q\alpha}{1-\alpha} d_{t}(1 + e_{t})$. 
Assuming the occurrence of event $\mathcal{\tilde S}$, as defined in equation \eqref{eq:def_tilde_s}, we get
\begin{eqnarray*}
\tilde s_{t}(u) &=& s_{t}(u) e^{-1/q}(1+ O(e_t + \sqrt{\frac{\psi}{s_t}} + \frac{1} {d_{t}})) \\
	&\geq& (1-\alpha)s_t e^{-1/q}(1+ O(e_t + \sqrt{\frac{\psi}{s_t}} + \frac{1} {d_{t}})) \\
	&\geq& (1-\alpha) s_{t+1}(1+ O(e_{t} + \sqrt{\frac{\psi}{s_t}} + \frac{1} {d_{t}})).
\end{eqnarray*}
Assuming event $\mathcal{\tilde A}$ occurs,
\begin{align*}
\tilde d_{t}(u) &\leq d_{t}(u)(1-\frac {q-1} {2q^3} \frac {s_{t}}{d_{t}} e^{-1/q})e^{-1/q} 
(1+  O(e_{t} + \sqrt{\frac {\psi}{s_{t}}} + \frac{1} {d_{t}})) \\
	&\leq \frac {1-q\alpha}{1-\alpha} d_{t}(1-\frac {q-1} {2q^3} \frac {s_{t}}{d_{t}} e^{-1/q})e^{-1/q} 
(1+  O(e_{t} + \sqrt{\frac {\psi}{s_{t}}} + \frac{1} {d_{t}})) \\
	&\leq \frac {1-q\alpha}{1-\alpha} \gamma d_{t+1}.
\end{align*}
where $\gamma$ is the smallest number in $[1,\infty)$ for which the above inequality is true.
Combining the preceding inequalities, we get
$$\gamma = 1+  O(e_{t} + \sqrt{\frac{\psi}{s_{t}}} + \frac{1} {d_{t}}).$$

In the cleanup phase of our algorithm(given in Section \ref{sec:algorithm_details}), 
the change in palette is equivalent to the following process.
\begin{enumerate}
\item
Add $\frac{\alpha}{1-\alpha} \tilde s_t(u)$ arbitrary colors to 
$u$'s palette, with $d_{u,c} = q \gamma d_{t+1}$. 
This adjusts the palette size to $ \tilde s_{t}(u) \geq s_{t+1}(1+ O(e_{t} + \sqrt{\psi/s_{t}} + 1/ d_t ))$.
Lemma \ref{lem:mu2} ensures that the adjusted new average is 
$ \tilde d_{t}(u) \leq \gamma d_{t+1}$
\item
Remove all the colors with $d_{t}(u,c) \geq q \gamma d_{t+1}$. 
\end{enumerate}
Now we use Lemma \ref{lem:mu1}, setting $\mu$ to $\gamma d_{t+1}$ and $q\mu$ to $q \gamma d_{t+1}$, to get
$$s_{t+1}(u) \geq (1-\alpha) s_{t+1}(1 + O(e_{t} + \sqrt{\frac{\psi}{s_t}} + \frac{1} {d_{t}}))$$
and 
$$d_{t+1}(u) \leq \frac{1-q \alpha}{1-\alpha} d_{t+1}(1+ O(e_{t} + \sqrt{\frac{\psi}{s_t}} + \frac{1} {d_{t}})).$$
The result is obtained using Lemmas \ref{lem:sg} and \ref{lem:tilde_a} to get
$$Pr(\mathcal{\tilde S} \cap \mathcal{\tilde A}) \geq 1 - e^{-\psi}O(n^2),$$
and noting that the preceding inequalities for $s_{t+1}(u)$ and $d_{t+1}(u)$ are true for every vertex $u$ in $V(G_{t+1})$,
given events $\mathcal{\tilde S}$ and $\mathcal{\tilde A}$ occur.

\qed
\end{proof}
Now consider the event
$$\mathcal{D} := \{ \forall v \in V(G_{t+1}), \eta_{t+1}(v) \leq \eta_{t+1}
	(1 + O(e_{t} + \sqrt{\frac {\psi}{d_{t}}} + \frac 1 {d_t}))\}.$$

\begin{lemma}
\label{lem:D}
Given Assumption \ref{ass:indhyp}, we have
$$ Pr(\mathcal{D}) \geq 1 - e^{-\psi} O(n).$$
\end{lemma}
\begin{proof}
Suppose $u$ is an uncolored vertex at the beginning of round $t$.
By Lemma \ref{lem:u_cl} we have
\begin{eqnarray*}
Pr\{\text{$u$ is colored}\} &\geq \frac {q-1} {2q^3} \frac {s_{t}} {d_{t}} e^{-1/q}
							(1+O(e_{t} + \frac 1 {d_t})).
\end{eqnarray*}
Using linearity of expectation, $\forall u$ in $V(G_{t+1})$
\begin{align*}
E[\eta_{t+1}(u)] 	&\leq \eta_{t}(u)(1 - \frac {q-1}{2q^3} \frac {s_{t}}{d_{t}} e^{-1/q})
					(1 + O(e_{t}+ \frac 1 {d_t})) & \langle \mbox{Lemma \ref{lem:O}}\rangle \\
			&\leq \eta_{t+1}(1 + O(e_{t} + \frac 1 {d_t})).
\end{align*}

We again resort to Theorem \ref{thm:azuma} to study concentration of $\eta_{t+1}(u)$.
Suppose $\eta_t(u) = m$.
Let $v_1, \dots, v_m$ be the vertices in $N_t(u)$.
Then $\eta_{t+1}(u)$ may be considered a random variable determined by
$T_1, \dots, T_m$, where $T_i$ is a random variable which indicates that $v_i$
is colored in round $t$ or not.
The affect of each $T_i$ given $T_1, \dots, T_{i-1}$ is at most $O(\frac {s_t}{d_t})$.
Thus $\sum \alpha_i^2$ in the statement of Theorem \ref{thm:azuma}
is
$$O(\frac {\eta_t s_t^2}{d_t^2}) = O(\frac{\eta_t s_t}{d_t}).$$
Using Theorem \ref{thm:azuma}, we get
$$ Pr\{\eta_{t+1}(u) \geq \eta_{t+1}(1 + O(e_{t} + \frac 1 {d_t})) 
	+ O(\sqrt{\psi \frac{s_{t} \eta_{t}}{d_{t}}})\} \leq e^{-\psi}.$$
Thus
$$ Pr\{\eta_{t+1}(u) \geq \eta_{t+1}(1 
	+ O(e_{t} + \sqrt{\psi \frac{s_{t}}{d_{t} \eta_{t}}} + \frac 1 {d_t}))\} \leq e^{-\psi}.$$
Since the above is true for every vertex $u$ in $V(G_{t+1})$,
the theorem is proved by applying the union bound on probabilties.
\qed
\end{proof}

\subsection{Expected Values and Concentration Inequalities for the Second Stage}

We now focus on the second stage of the algorithm, where $s_{t} \geq q^2 d_{t}$.
The pattern of analysis for the second stage is similar to that of the first stage.
But simply reproducing the previous section, with changes here and there,
would make it difficult to understand the differences.
With this in mind, we proceed much faster now and focus on the
expectations of variables, leaving out all concentration calculations.
These can be filled in by using the corresponding proofs in 
Section \ref{sec:first_stage} as templates.

Consider now the event
$$ \mathcal{\tilde S} := \{\forall u \in V(G_{t+1}), s_{t}(u) e^{-\frac 1 {q} \frac {d_{t}}{s_{t}}}
		(1 - O(e_t + \sqrt{\frac{\psi}{s_{t}}} + \frac{1}{s_{t}})) 
	\leq \tilde s_{t}(u) \leq s_{t}(u) e^{-\frac 1 {q} \frac {d_{t}}{s_{t}}}
		(1 + O(e_t + \sqrt{\frac{\psi}{s_{t}}} + \frac{1}{s_{t}}))\}.$$

\begin{lemma}
\label{lem:2sg}
Given Assumption \ref{ass:indhyp}, we have
$$ Pr(\mathcal{\tilde S}) \geq 1- e^{-\psi}O(n).$$
\end{lemma}

\begin{proof}
Suppose $u$ is an uncolored vertex at the beginning of round $t$, and $c$ a color in its palette.
\begin{eqnarray*}
Pr\{\text{$c$ is removed from $S_t(u)$ in phase II.1}\} 
	&=& 1 - \prod_{v \in D_{t}(u,c)}(1 - Pr\{ \text{$v$ is assigned $c$} \}) \\
	&=& 1 - \prod_{v \in D_{u,c}}(1 - \frac 1 {q^2} \frac 1 {s_{t}} (1+O(e_{t} + \frac{1}{s_{t}}))) \\
	&\leq& 1 - e^{- \frac 1 {q} \frac {d_{t}}{s_{t}}}(1 + O(e_{t} + \frac {1}{s_{t}}))
\end{eqnarray*}
In phase II.2 we remove colors from the palette using an appropriate bernoulli variable, to get
$$Pr\{c \notin \tilde S_t(u)\} = 
	1 - e^{-\frac 1 {q} \frac {d_{t}}{s_{t}}}(1 + O(e_{t} + \frac {1}{s_{t}})).$$
Using linearity of expectation
$$ \forall u \in V(G_{t+1}), E[\tilde s_{t}(u)] = s_{t}(u) e^{-\frac 1 {q} \frac {d_{t}}{s_{t}}}
									(1 + O(e_{t} + \frac {1}{s_{t}})).$$

The rest of the proof follows that of Lemma \ref{lem:sg}.
\qed
\end{proof}
We now focus on the sets $D_t(u,c)$. The following two lemmas will help.
\begin{lemma}
\label{lem:u_cl_c_2}
Let $u$ be an uncolored vertex at the beginning of round $t$, and $c$ a color in its palette.
Then given Assumption \ref{ass:indhyp}, we have
$$Pr\{\text{$u$ is assigned $c$ and $c \in \tilde S_t(u)$}\} 
	\geq \frac 1 {q^2} \frac 1 {s_{t}} e^{- \frac 1 q \frac {d_t}{s_t}}
				(1+O(e_{t} + \frac {1}{s_{t}})).$$
\end{lemma}
\begin{proof}
\begin{eqnarray*}
	Pr\{\text{$u$ is assigned $c$ and $c \in \tilde S_t(u)$}\}
		&=& Pr\{ \text{$u$ is assigned $c$}\} Pr\{ c \in \tilde S_t(u)\} \\
		&=& \frac 1{q^2} \frac{1}{s_{t}} \prod_{v \in D_{t}(u,c)} (1 - \frac 1{q^2}\frac{1}{s_{t}}) \\
		&\geq& \frac 1 {q^2} \frac 1 {s_{t}} e^{-\frac 1 {q} \frac {d_{t}}{s_{t}}}
									(1+O(e_{t} + \frac{1}{s_{t}}))
\end{eqnarray*}
\qed
\end{proof}
The following lemma is a consequence of the previous one.
\begin{lemma}
\label{lem:u_cl_2}
Let $u$ be an uncolored vertex at the beginning of round $t$. 
Then given Assumption \ref{ass:indhyp}, we have
$$Pr\{u\text{ is colored}\} 
	\geq \frac {q-1}{2 q^3} e^{-\frac 1 q \frac {d_t}{s_t}}(1+O(e_{t} + \frac {1}{s_{t}})).$$
\end{lemma}
\begin{proof}
Following the proof of Lemma \ref{lem:u_cl}, but using Lemma \ref{lem:u_cl_c_2} instead of Lemma \ref{lem:u_cl_c}, we get
\begin{align*}
Pr \{u \mbox{ is colored}\} 
	&\geq 1 - (1 - \frac 1 {q^2} \frac 1 {s_t} e^{\frac 1 q \frac {d_t}{s_t}}
		(1 + O(e_t + \frac 1 {s_t})))^{s_t(u)} \\
	&\geq  \frac {q-1} {2q^3} e^{-\frac 1 q \frac {d_t} {s_t}}(1 + O(e_t + \frac 1 {s_t})).
\end{align*}
\qed
\end{proof}
Consider the event
$$\mathcal{\tilde A} := \{\forall u \in V(G_{t+1}), \tilde d_{t}(u) 
\leq d_{t}(u) (1 - \frac {q-1}{2q^3} e^{-\frac 1 {q} \frac {d_{t}}{s_{t}}})
	e^{-\frac 1 {q} \frac {d_{t}}{s_{t}}}
	(1 + O(e_{t} + \sqrt{\frac{\psi} {s_{t}}} + \frac{1}{s_{t}} + \sqrt{\frac{\psi d_{t}}{s_{t} d_{t}(u)}}))\}.$$

\begin{lemma}
Given Assumption \ref{ass:indhyp}, we have
$$Pr(\mathcal{\tilde A}) \geq 1 - e^{-\psi}O(n^2).$$
\end{lemma}
\begin{proof}
Let $u$ be an uncolored vertex at the beginning of round $t$,
and let $c$ be a color in its palette.
For a vertex $v$ in $D_{t}(u,c)$, the event
\begin{eqnarray*}
\{v \notin \tilde D_t(u,c) \} = \{\text{$v$ is colored}\} \cup (\{ \text{$v$ is not colored}\} \cap 
\{c \notin \tilde S_t(v)\}).
\end{eqnarray*}
Then as in the proof of Lemma \ref{lem:tilde_a}, we get
\begin{eqnarray*}
Pr\{v \notin \tilde D_t(u,c) | c \in \tilde S_t(u)\} &=& Pr\{v \notin \tilde D_t(u,c)\}(1+O(\frac 1 {s_t})) \\
	&=& (Pr\{\text{$v$ is colored}\} + (1 - e^{-\frac 1 {q} \frac {d_{t}}{s_{t}}})
			(1 - Pr\{\text{$v$ is colored}\}))(1 + O(\frac {1}{s_{t}})) \\
	&=& (1 - (1 - Pr\{\text{$v$ is colored}\}) e^{-\frac 1 {q} \frac {d_{t}}{s_{t}}})
		(1+ O(e_{t} + \frac {1}{s_{t}})) \\
	&\geq& (1 - (1 - \frac {q-1} {2q^3} e^{-\frac 1 {q} \frac {d_{t}}{s_{t}}}) e^{-\frac 1 {q} \frac {d_{t}}{s_{t}}})(1+ O(e_{t} + \frac {1}{s_{t}})).
\end{eqnarray*}
By linearity of expectation, $\forall u \in V(G_{t+1}), \forall c \in \tilde s_{t}(u)$
\begin{equation*}
E[\tilde d_{t}(u,c) | c \in \tilde s_{t}(u)] 
\leq E[\tilde d_{t}(u,c)](1+O(\frac 1{s_t})) 
\leq d_{t}(u,c)(1 - \frac {q-1} {2q^3} e^{-\frac 1 {q} \frac {d_{t}}{s_{t}}}) e^{-\frac 1 {q} \frac {d_{t}}{s_{t}}}(1+ O(e_{t} + \frac {1}{s_{t}})).
\end{equation*}
Now using the above bound
\begin{eqnarray*}
E[\bar d_{t}(u)] &=& \sum_{c \in s_{t}(u)} Pr\{c \in \tilde s_{t}(u)\} E[\tilde d_{t}(u,c) | c \in \tilde s_{t}(u)] \\
	&\leq& e^{-\frac 1 {q} \frac {d_{t}}{s_{t}}} \sum_{c \in s_{t}(u)} d_{t}(u,c)(1 -\frac {q-1} {2q^3} e^{-\frac 1 {q} \frac {d_{t}}{s_{t}}}) e^{-\frac 1 {q} \frac {d_{t}}{s_{t}}}(1+ O(e_{t} + \frac {1}{s_{t}})) \\
	&\leq& e^{-\frac 1 {q} \frac {d_{t}}{s_{t}}} s_{t}(u) d_{t}(u) (1 - \frac {q-1} {2q^3} e^{-\frac 1 {q} \frac {d_{t}}{s_{t}}}) e^{-\frac 1 {q} \frac {d_{t}}{s_{t}}}(1+ O(e_{t} + \frac {1}{s_{t}})).
\end{eqnarray*}
The rest of the proof follows that of Lemma \ref{lem:tilde_a}.
\qed
\end{proof}
Consider the event
$$\mathcal{A} := \{\forall u \in V(G_{t+1}) d_{t+1}(u) \leq d_{t+1}(1 +O( e_{t} + \sqrt{\frac{\psi} {s_{t}}} + \frac{1}{s_{t}})\}.$$

\begin{lemma}
\label{lem:2A}
Given Assumption \ref{ass:indhyp}, we have
$$Pr(\mathcal{A}) \geq 1 - e^{-\psi} O(n^2).$$
\end{lemma}
\begin{proof}
The proof follows that of Lemma \ref{lem:A}.
\qed
\end{proof}
Consider the event
\begin{align*}
\mathcal{S} := 
\{
\forall v \in V(G_{t+1}), & \exists \alpha \in[0, \frac 1 q] \text{ such that } \\
	 & s_{t+1}(u) \geq (1-\alpha) s_{t+1}
	 		(1 + O(e_{t} + \sqrt{\frac{\psi} {s_{t}}} + \frac{1} {s_{t}})), \\
	& d_{t+1}(u) \leq \frac{1-q \alpha}{1-\alpha} d_{t+1}(1+ O(e_{t} + \sqrt{\frac{\psi} {s_{t}}} + \frac{1} {s_{t}}))
	\}.
\end{align*}
\begin{lemma}
\label{lem:2S}
Given Assumption \ref{ass:indhyp}, we have
$$Pr(\mathcal{S}) \geq 1 - e^{-\psi} O(n^2).$$
\end{lemma}

\begin{proof}
The proof follows that of Lemma \ref{lem:S}.
\qed
\end{proof}
Now consider the event
$$\mathcal{D} := \{ \forall v \in V(G_{t+1}), \eta_{t+1}(v) \leq \eta_{t+1}(1 + O(e_{t} + \sqrt{\frac {\psi}{d_{t}}}))\}.$$

\begin{lemma}
\label{lem:2D}
Given Assumption \ref{ass:indhyp}, we have
$$Pr(\mathcal{D}) \geq 1 - e^{-\psi} O(n)$$
\end{lemma}
\begin{proof}
Using Lemma \ref{lem:u_cl_2} 
\begin{eqnarray*}
Pr\{\text{$u$ is colored}\}  
	&\geq \frac {q-1} {2q^3} e^{-\frac 1 {q} \frac {d_{t}}{s_{t}}}(1+O(e_{t})).
\end{eqnarray*}
Using linearity of expectation, $\forall u$ in $G_{t+1}$
\begin{align*}
E[\eta_{t+1}(u)] 	&\leq \eta_{t}(u)(1 - \frac {q-1} {2q^3} e^{-\frac 1 {q} \frac {d_{t}}{s_{t}}}(1 + O(e_{t})) \\
			&\leq \eta_{t+1}(1 + O(e_{t})).
\end{align*}
The rest of the proof follows that of Lemma \ref{lem:D}.
\qed
\end{proof}

\section{Acknowledgement}
I am indebted to Fan Chung  and Jacques Verstraete for their comments and support.

\bibliographystyle{amsplain}
\bibliography{color}

\providecommand{\bysame}{\leavevmode\hbox to3em{\hrulefill}\thinspace}
\providecommand{\MR}{\relax\ifhmode\unskip\space\fi MR }
\providecommand{\MRhref}[2]{%
  \href{http://www.ams.org/mathscinet-getitem?mr=#1}{#2}
}
\providecommand{\href}[2]{#2}
\begin{thebibliography}{10}

\bibitem{AM97}
Dimitris Achlioptas and booktitle = {FOCS} year = {1997} pages = {204-212}
  bibsource =~{DBLP, http://dblp.uni-trier.de} Michael~Molloy, title =~{The
  analysis of a list-coloring algorithm on a random graph}.

\bibitem{AKS81}
Mikl{\'o}s Ajtai, J{\'a}nos Koml{\'o}s, and Endre Szemer{\'e}di, \emph{A dense
  infinite {S}idon sequence}, European J. Combin. \textbf{2} (1981), no.~1,
  1--11. \MR{611925 (83f:10056)}

\bibitem{AKS99}
Noga Alon, Michael Krivelevich, and Benny Sudakov, \emph{Coloring graphs with
  sparse neighborhoods}, J. Combin. Theory Ser. B \textbf{77} (1999), no.~1,
  73--82. \MR{1710532 (2001a:05054)}

\bibitem{AS08}
Noga Alon and Joel~H. Spencer, \emph{The probabilistic method}, third ed.,
  Wiley-Interscience Series in Discrete Mathematics and Optimization, John
  Wiley \& Sons Inc., Hoboken, NJ, 2008, With an appendix on the life and work
  of Paul Erd{\H{o}}s. \MR{2437651 (2009j:60004)}

\bibitem{B78}
B{\'e}la Bollob{\'a}s, \emph{Chromatic number, girth and maximal degree},
  Discrete Math. \textbf{24} (1978), no.~3, 311--314. \MR{523321 (80e:05058)}

\bibitem{BK77}
O.~V. Borodin and A.~V. Kostochka, \emph{On an upper bound of a graph's
  chromatic number, depending on the graph's degree and density}, J.
  Combinatorial Theory Ser. B \textbf{23} (1977), no.~2-3, 247--250.
  \MR{0469803 (57 \#9584)}

\bibitem{B41}
R.~L. Brooks, \emph{On colouring the nodes of a network}, Proc. Cambridge
  Philos. Soc. \textbf{37} (1941), 194--197. \MR{0012236 (6,281b)}

\bibitem{C78}
Paul~A. Catlin, \emph{A bound on the chromatic number of a graph}, Discrete
  Math. \textbf{22} (1978), no.~1, 81--83. \MR{522914 (80a:05090a)}

\bibitem{FR85}
P.~Frankl and V.~R{\"o}dl, \emph{Near perfect coverings in graphs and
  hypergraphs}, European J. Combin. \textbf{6} (1985), no.~4, 317--326.
  \MR{829351 (88a:05116)}

\bibitem{GJ79}
Michael~R. Garey and David~S. Johnson, \emph{Computers and intractability: A
  guide to the theory of np-completeness}, W.H. Freeman, New York, 1979.

\bibitem{GP00}
David~A. Grable and Alessandro Panconesi, \emph{Fast distributed algorithms for
  {B}rooks-{V}izing colorings}, J. Algorithms \textbf{37} (2000), no.~1,
  85--120, Ninth Annual ACM-SIAM Symposium on Discrete Algorithms (San
  Francisco, CA, 1998). \MR{1783250 (2002d:05115)}

\bibitem{J96}
A.~Johansson, \emph{Asymptotic choice number for triangle free graphs},
  Unpublished, see Molloy and Reed \cite{MR01}.

\bibitem{K92}
Jeff Kahn, \emph{Coloring nearly-disjoint hypergraphs with {$n+o(n)$} colors},
  J. Combin. Theory Ser. A \textbf{59} (1992), no.~1, 31--39. \MR{1141320
  (93b:05127)}

\bibitem{K96}
\bysame, \emph{Asymptotically good list-colorings}, J. Combin. Theory Ser. A
  \textbf{73} (1996), no.~1, 1--59. \MR{1367606 (96j:05001)}

\bibitem{K01}
Subhash Khot, \emph{Improved inaproximability results for maxclique, chromatic
  number and approximate graph coloring}, FOCS, 2001, pp.~600--609.

\bibitem{K95}
Jeong~Han Kim, \emph{On {B}rooks' theorem for sparse graphs}, Combin. Probab.
  Comput. \textbf{4} (1995), no.~2, 97--132. \MR{1342856 (96f:05078)}

\bibitem{K2_95}
\bysame, \emph{The {R}amsey number {$R(3,t)$} has order of magnitude {$t^2/\log
  t$}}, Random Structures Algorithms \textbf{7} (1995), no.~3, 173--207.
  \MR{1369063 (96m:05140)}

\bibitem{KM77}
A.~V. Kosto{\v{c}}ka and N.~P. Masurova, \emph{An estimate in the theory of
  graph coloring}, Diskret. Analiz (1977), no.~30 Metody Diskret. Anal. v
  Resenii Kombinatornyh Zadac, 23--29, 76. \MR{0543805 (58 \#27604)}

\bibitem{L78}
Jim Lawrence, \emph{Covering the vertex set of a graph with subgraphs of
  smaller degree}, Discrete Math. \textbf{21} (1978), no.~1, 61--68. \MR{523419
  (80a:05094)}

\bibitem{MR01}
Michael Molloy and Bruce Reed, \emph{Graph colouring and the probabilistic
  method}, Algorithms and Combinatorics, vol.~23, Springer-Verlag, Berlin,
  2002. \MR{1869439 (2003c:05001)}

\bibitem{PS89}
Nicholas Pippenger and Joel Spencer, \emph{Asymptotic behavior of the chromatic
  index for hypergraphs}, J. Combin. Theory Ser. A \textbf{51} (1989), no.~1,
  24--42. \MR{993646 (90h:05091)}

\bibitem{V68}
V.~G. Vizing, \emph{Some unsolved problems in graph theory}, Uspehi Mat. Nauk
  \textbf{23} (1968), no.~6 (144), 117--134. \MR{0240000 (39 \#1354)}

\bibitem{V02}
Van~H. Vu, \emph{A general upper bound on the list chromatic number of locally
  sparse graphs}, Combin. Probab. Comput. \textbf{11} (2002), no.~1, 103--111.
  \MR{1888186 (2003c:05090)}

\bibitem{W95}
Nicholas~C. Wormald, \emph{Differential equations for random processes and
  random graphs}, Ann. Appl. Probab. \textbf{5} (1995), no.~4, 1217--1235.
  \MR{1384372 (97c:05139)}

\end{thebibliography}

\end{document}